\documentclass[a4paper,12pt]{amsart}

\usepackage[utf8]{inputenc}
\usepackage[OT4]{fontenc}
\usepackage{epsf}
\usepackage{hyperref}

\usepackage{amsmath, amssymb, amsthm}

\newcommand{\PP}{\mathbb{P}}
\newcommand{\KK}{\mathbb{K}}
\newcommand{\NN}{\mathbb{N}}

\newcommand{\QQ}{\mathbb{Q}}

\newcommand{\cL}{\mathcal{L}}

\newcommand{\s}{\subset}

\DeclareMathOperator{\id}{id}

\newtheorem{twierdzenie}{Theorem}
\newtheorem{lemat}[twierdzenie]{Lemma}
\newtheorem{propozycja}[twierdzenie]{Proposition}
\newtheorem{obserwacja}[twierdzenie]{Observation}
\theoremstyle{definition}
\newtheorem{definicja}[twierdzenie]{Definition}
\newtheorem{przyklad}[twierdzenie]{Example}
\theoremstyle{remark}
\newtheorem{uwaga}[twierdzenie]{Remark}

\title[Linear systems over $\PP^1\times\PP^1$]
{Linear systems over $\PP^1\times\PP^1$ with base
points of multiplicity bounded by three}
\date{}
\author{Tomasz Lenarcik}

\begin{document}

\maketitle

\begin{abstract}
We propose a combinatorial method of proving non-specialty
of a linear system of curves with multiple points in general positions.
As an application we obtain a classification of special linear
systems on $\PP^1\times\PP^1$ for which the multiplicities do not
exceed~$3$.
\end{abstract}

\section*{Introduction}

Let $p_1,\ldots,p_r$ denote points in $\PP^1\times\PP^1$ in general
position and let $m_1,\ldots,m_r$ be positive integers. Consider
a blowing up $\pi:X\longrightarrow\PP^1\times\PP^1$ at $p_1,\ldots,p_r$,
and denote the exeptional divisors respectively by $E_1,\ldots,E_r$.
For given $d,e\geq 0$ we define $\cL_{(d,e)}(p_1 m_1,\ldots,p_r m_r)$
to be a complete linear system of the following divisor:
$$
dH_1+eH_2-m_1E_1-\ldots-m_rE_r
$$
where $H_1$ and $H_2$ are pullbacks of classes
of $\PP^1\times\{a_1\}$ and $\{a_2\}\times\PP^1$
respectively, and $a_1,a_2\in\PP^1$ are arbitrary.
It can be understood as a linear space of curves of bidegree $(d,e)$
that vanish at $p_i$ with a multiplicity of at least $m_i$ for $i=1,\ldots,r$.
For a sufficiently general choice of affine coordinates,
each curve from $\cL_{(d,e)}(p_1 m_1,\ldots,p_r m_r)$ can be
uniquely represented (up to a constant factor) by a polynomial
in two variables, namely $X$ and $Y$, that contains monomials
of the form $X^{\alpha}Y^{\beta}$, where $0\leq\alpha\leq d$
and $0\leq\beta\leq e$. Therefore, from linear algebra it follows that
the projective dimension of $\cL_{(d,e)}(p_1 m_1,\ldots,p_r m_r)$
is not less than:
\begin{equation}\label{eq_edim}
\max\{-1,(d+1)(e+1)-\sum_{i}\binom{m_i+1}{2}-1\}.
\end{equation}
The actual dimension, however, does not have to equal
the \textit{expected dimension} \eqref{eq_edim},
as the equations may happen to be
linearly dependent even for a general choice of $p_1,\ldots,p_r$.
In such an instance, we say that the linear system is \textit{special}.
A similar definition can be formulated for special linear
systems over $\PP^2$ (see for example \cite{Dum}).
Further information about linear systems shall be presented in Section 1.

In Section 2 we propose a combinatorial technique of proving
non-specialty of a linear system. Several approaches of this kind
have been recently developed including degenerations techniques
by Ciliberto, Dumitrescu and Miranda \cite{CDM},
application of tropic geomentry by Baur and Draisma \cite{Dra,Bau-Dra},
and reduction method by Dumnicki and Jarnicki \cite{Dum, Dum-Jar}.

Comprehensive research has been done on linear systems over $\PP^2$
due to the Gimigliano-Harbourne-Hirschowitz Conjecture
(see \cite{Hir} for orginal statement or \cite{Dum} for further references).
In \cite{Dum-Jar}, Dumnicki and Jarnicki gave a classification of
all special systems over $\PP^2$ with multiplicities bounded by $11$,
which made it possible to verify that
Gimigliano-Harbourne-Hirschowitz Conjecture
holds for all systems of this type. The case of $\PP^1\times\PP^1$
seems to be less developed in terms of such classification.
As long as all multiplicities of the base points equal $2$
the problem of specialty of a linear system has been widely studied
by many authors for varieties of type
$\PP^{n_1}\times\ldots\times\PP^{n_k}$.
This is due to the fact that special linear systems of this kind
are closely related to defective Segre-Veronese embeddings
(see \cite{CGG1, CGG2} by Catalisano, Geramita, Gimigliano
and \cite{Bau-Dra, CDM}).

As an application of the method introduced in Section 2, we state
Theorem \ref{main1}, which gives a characterization of
special linear systems over $\PP^1\times\PP^1$ with base points
of multiplicity bounded by $3$. While writing this paper
we found that such characterization had already been
known for homogenous systems, i.e. systems with base points of all
the multiplicities equal to 3 (Laface \cite{Laf}). The proof of
Theorem \ref{main1}, which is the main result of this paper,
shall be presented in Section 3.

\section{Linear systems}
Let $\KK$ be an arbitrary field, $\NN=\{0,1,2,\ldots\}$,
$\NN^*=\{1,2,3,\ldots\}$. For any $\delta\in\NN^2$ we
write $\delta=(\delta_1,\delta_2)$.
\begin{definicja}\label{def01}
Any finite and non-empty set $D\s\NN^2$ shall be called a \textit{diagram}.
Let $r\in\NN^*$ and $m_1,\ldots,m_r\in\NN$
(there is a technical reason to consider zero,
see for example Definition \ref{def_small_matrix}).
Let $L$ be a field of rational functions over $K$ in variables
$x_1,y_1,\ldots,x_r,y_r$. Define a \textit{linear system}
spanned over a diagram $D$ with base points of multiplicities
$m_1,\ldots,m_r$ to be an $L$-vector space
$\cL_D(m_1,\ldots,m_r)\s L[X,Y]$ of polynomials
$f=\sum_{\delta\in D} A_{\delta}X^{\delta_1}Y^{\delta_2}$, such that:
\begin{equation}\label{eq01}
\frac{\partial^{\alpha+\beta}f}
{\partial X^{\alpha}\partial Y^{\beta}}(x_i,y_i)=0,\text{ for }
i=1,\ldots,r\text{ and } \alpha+\beta<m_i
\end{equation}
Let $M=M_D(m_1,\ldots,m_r)$ be the matrix of the system of equations
\eqref{eq01}, which are linear with respect to unknown coefficients
$\{A_{\delta}\}_{\delta\in D}$. We say that the system $\cL_D(m_1,\ldots,m_r)$
is special when $M$ is not of the maximal rank. Observe that
the enteries of $M$ belong to the polynomial ring
$\KK[x_1,y_1,\ldots,x_r,y_r]$.
\end{definicja}

Given $d,e\in\NN$ we denote by $\cL_{(d,e)}(m_1,\ldots,m_r)$
a linear system spanned over the diagram
$\{0,1,\ldots,d\}\times\{0,1,\ldots,e\}$.
Let $D$ be a diagram and $m_1,\ldots,m_t\in\NN$,
$q_1,\ldots,q_t\in\NN^*$. We shall use the following notation:
$$
\cL_D(m_1^{\times q_1},\ldots,m_t^{\times q_t}):=
\cL_D(\overbrace{m_1,\ldots,m_1}^{q_1},\ldots,
\overbrace{m_t,\ldots,m_t}^{q_t})
.
$$

\begin{twierdzenie}\label{main1}
Let us assume that $0\leq d\leq e$. A linear system of the form
$\cL_{(d,e)}(1^{\times p},2^{\times q},3^{\times r})$
is special if and only if one of the following conditions hold:
\begin{itemize}
 \item[$(0)$] $d=0$, $p+2q+3r\leq e$ and $q\geq 1$ or $r\geq 1$
 \item[$(1)$] $d=1$, $p+3q+5r\leq 2e+1$ and $r\geq 1$
 \item[$(2)$] $d=2$, $p=0$, $e=q+2r-1$ and $2\nmid q+r$
 \item[$(3)$] $d=3$, and for some $n\geq 1$:
 \begin{itemize}
  \item[$(3.1)$] $e=3n$, $p=q=0$ and $r=2n+1$
  \item[$(3.2)$] $e=3n$, $p\leq 1$, $q=1$ and $r=2n$
  \item[$(3.3)$] $e=3n+1$, $p\leq 2$, $q=0$ and $r=2n+1$
  \item[$(3.4)$] $e=3n+2$, $p=0$, $q=2$ and $r=2n+1$
 \end{itemize}
 \item[$(4)$] $d=4$, $e=5$, $p=q=0$ and $r=5$
\end{itemize}
\end{twierdzenie}

The proof of Theorem \ref{main1} shall be presented in Section 3.
Throughout the proof we will take advantage of
a relation between geometrical properties of diagram $D$
and the rank of matrix $M_D(m_1,\ldots,m_r)$ (see Theorem \ref{prop08}).
For linear systems which contain only one base point, this relation
can be expressed as follows (the proof can be found in \cite{Dum}):
\begin{propozycja}\label{prop02}
Let $D=\{\delta_1,\ldots,\delta_s\}$ be a diagram, and
$\#D=s=\binom{m_1+1}{2}$. Then $\det M_D(m_1)=0$
(i.e. sytem $\cL_D(m_1)$ is special) if and only if
there extists a curve of degree $m_1-1$ that contains
all of the points of $D$.
Moreover, if $\det M_D(m_1)\neq 0$ then the following equation holds:
$$
\det M_D(m)=A\cdot x_1^{\delta_{1,1}+\ldots+\delta_{s,1}-s}
y_1^{\delta_{1,2}+\ldots+\delta_{s,2}-s},
\quad A\in\KK\setminus0
$$
\end{propozycja}

\begin{definicja}\label{def03}
Let $D$ be a diagram and $\#D=\binom{m+1}{2}$.
We say that $D$ is \textit{non-special (special)} of degree $m$
if $\det M_D(m)\neq 0(=0)$.
\end{definicja}


\section{Unique tilings}

We introduce the notion of a \textit{unique tiling} and state
Theorem \ref{prop08} and Theorem \ref{tw13}. Thanks to these theorems
we will be able to prove non-specialty of linear systems in terms of
finding a solution for some specific problem of exact covering.

\begin{definicja}[center of mass, unique tiling]\label{def04}
Given a diagram $D$ we define its \textit{center of mass}
as follows:
\begin{equation}
\QQ^2\ni(c_1(D),c_2(D))=c(D):=\frac{1}{\#D}\sum_{d\in D}d
\end{equation}

We say that a finite and non-empty set of diagrams $T$
is \textit{a tiling}, if any two elements of $T$ are disjoint.
Let $T$ and $T'$ be tilings, and consider a mapping
$f:T\longrightarrow T'$. We say that $T$ and $T'$ are congruent
through $f$, and denote it by $f:T\simeq T'$,
if the following conditions hold:
\begin{itemize}
 \item[$(i)$] $f$ is one to one, and $\#f(D)=\#D$, $c(f(D))=c(D)$
  for any $D\in T$ 
 \item[$(ii)$] $\bigcup T=\bigcup T'$
\end{itemize}
A tiling $T$ that contains only non-special diagrams is said
to be \textit{unique}, if $f:T\simeq T'$ implies that either
$T=T'$ and $f=\id_U$, or $T'$ contains a special diagram.
\end{definicja}

\begin{definicja}\label{def_small_matrix}
Given a diagram $D$ and numbers $m,r,i>0$, where $i\leq r$,
we define:
$$
M_D^{(i)}(m):=M_D(0^{\times i-1},m,0^{\times {r-i}})
$$
Observe, that the entries of the matrix $M_D^{(i)}(m)$
depend only on variables $x_i,y_i$ (see Definition \ref{def01}).
\end{definicja}

The following theorem states a relation between
uniqueness of a particular tiling and non-specialty
of a linear system.
\begin{twierdzenie}\label{prop08}
Let $T=\{D_1,\ldots,D_r\}$ be a unique tiling such that
$\#D_i=\binom{m_i+1}{2}$. If $D$ is a diagram for which
one of the following conditions holds:
\begin{itemize}
 \item[$(i)$] $D_1\cup\ldots\cup D_r\s D$
 \item[$(ii)$] $D\s D_1\cup\ldots\cup D_r$
\end{itemize}
then $\cL_D(m_1,\ldots,m_r)$ is non-special.
\begin{proof}
Let us denote $M=M_D(m_1,\ldots,m_r)$. We shall group
the rows of $M$ into submatricies $M_1,\ldots,M_r$
such that $M_i$ coresponds to $\binom{m_i+1}{2}$ equations
which depend on variables $x_i,y_i$ (see Definition \ref{def01}).
The columns of $M$ are indexed by the elements of $D$ in a natural way
(each element of $D$ coresponds to a monomial).

It is sufficient to prove the theorem under the assumption of $(i)$.
Let $D'=D_1\cup\ldots\cup D_r$. From $(i)$ it follows that the
minor $M(D')$, i.e. the submatrix of $M$ consisting of columns indexed
by the elements of $D'$, is of the maximal rank.
By the Laplace decomposition we get:
\begin{multline}\label{eq02}
\det M(D')
=\sum \varepsilon(D_1',\ldots,D_r')
\det M_1(D_1')\cdots\det M_r(D_r') \\
=\sum \varepsilon(D_1',\ldots,D_r')
\det M^{(1)}_{D_1'}(m_1)\cdots\det M^{(r)}_{D_r'}(m_r)
\end{multline}
where $\varepsilon(D_1',\ldots,D_r')\in\{-1,1\}$
and the summation runs over all partitions $\{D_1',\ldots,D_r'\}$
of $D'$ for witch $\#D_i'=\binom{m_i+1}{2}$.

Let $s_i=\#D_i'$. From Proposition \ref{prop02} it follows
that each non-zero component of the sum \eqref{eq02} is
non-zero if and only if, $D_1',\ldots,D_r'$ are non-special.
In which case this component is a monomial of the form:
\begin{multline}\label{eq03}
\det M^{(1)}_{D_1'}(m_1)\cdots\det M^{(r)}_{D_r'}(m_r)\\
=A\cdot
x_1^{c_1(D_1')s_1-s_1}
y_1^{c_2(D_1')s_1-s_1}
\cdots
x_r^{c_1(D_r')s_r-s_r}
y_r^{c_2(D_r')s_r-s_r}
\end{multline}
where $A\in\KK$. As the tiling $T$ is unique, the non-zero monomial:
$$
\det M^{(1)}_{D_1}(m_1)\cdots\det M^{(r)}_{D_r}(m_r)
$$
turnes up as a component of \eqref{eq02} only once.
Since it can not be reduced we get $\det M(D')\neq 0$.
\end{proof}
\end{twierdzenie}

A weak point of Theorem \ref{prop08} is its assumption
about the uniqueness of the tiling. Verifying whether
a given tiling is unique or not may seem to be even more
challanging than evaluating ''by hand'' the rank of the matrix
related to a linear system. This problem is partialy addressed
by Theorem \ref{tw13}, which aims at giving some conditions
that are sufficient for the uniqueness of a tiling.
Before we state this result, we need to introduce some
necessary definitions.

\begin{definicja}[innertial momentum, stable diagram]\label{def09}
Given a diagram $D$ we define its \textit{boundary distributions} and
\textit{innertial momentum} as follows:
\begin{gather*}
\varphi_1(D):\NN\ni\alpha\mapsto\#(D\cap\{\alpha\}\times\NN)\in\NN\\
\varphi_2(D):\NN\ni\beta\mapsto\#(D\cap\NN\times\{\beta\})\in\NN\\
i(D)=\sum_{\delta\in D}\Vert \delta-c(D)\Vert^2
\end{gather*}
We say that a diagram $D$ is \textit{a stable diagram} if it is
non-special, its vertical and horizontal sections are segments,
and for any diagram $D'$ the equations $\#D=\#D'$ and
$c(D)=c(D')$ imply that at least one of the following conditions holds:
\begin{itemize}
 \item[$(i)$] $D'$ is special
 \item[$(ii)$] $i(D')>i(D)$
 \item[$(iii)$] $i(D')=i(D)$, $\varphi_1(D')=\varphi_1(D)$
     and $\varphi_2(D')=\varphi_2(D)$
\end{itemize}
\begin{przyklad}\label{ex_sing}
Any $1-$diagram, i.e. a diagram which is a singleton,
is a stable diagram.
\end{przyklad}
We define a relation on the set of all diagrams:
\begin{multline*}
D\preceq D'\iff 
\text{there exists }\delta=(\delta_1,\delta_2)\in D\text{ and }
\delta'=(\delta_1',\delta_2')\in D',\\
\text{such that }\delta_1=\delta_1'\text{ and }\delta_2\preceq \delta_2'
\end{multline*}
When restricted to a particular tiling, this relation may be
extended to a partial ordering. The following observation
contains further details. We omit the simple proof.
\begin{obserwacja}\label{obs12}
Suppose that the projection on the first coordinate
of any diagram from a tiling $T$ is a segment.
Then the following conditions are equivalent:
\begin{itemize}
 \item[$(i)$] relation $\preceq$ can be extended to a partial ordering on $T$
 \item[$(ii)$] for any $D,D'\in T$ the relations $D\preceq D'$ and
 $D'\preceq D$ imply $D=D'$.
\end{itemize}
\end{obserwacja}
\end{definicja}

\begin{twierdzenie}\label{tw13}
Suppose that a tiling $T$ consists of stable diagrams. If the relation
$\preceq$ can be extended to a partial ordering on $T$ then $T$ is unique.
\end{twierdzenie}

\begin{uwaga}
Whether a given diagram is special or not, can be usually verified
with the help of B\'ezout Theorem due to Proposition \ref{prop02}.
Thanks to Observation \ref{obs12}$(ii)$ it is very easy to state,
for a given tiling, that $\preceq$ can be extended to a partial ordering.
Meanwhile, determining if a diagram is stable or not seems to be a more
complex task. As the condition of being a stable diagram is an indeterminant
of an isometry, the problem of finding all stable diagrams of a bounded 
degree leads to a finite number of cases, and so it can be solved
through effective, but harmfull, computation.

The figure below represents all stable diagrams, up to an isometry,
consisting of $3$ or $6$ elements. Every diagram that is isometric
to one of the following shall be called either \textit{3-diagram}
or \textit{6-diagram}.
\begin{figure}[h]
\epsfbox{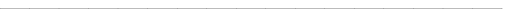}
\caption{Examples of $3-$diagrams and $6-$diagrams}
\end{figure}
\end{uwaga}

\begin{proof}[Proof of Theorem \ref{tw13}]
We follow by induction on the number of elements of $T$.
It is clear that a tiling consisting of one diagram is always unique.

Let us denote $T=\{D_1,\ldots,D_s\}$ and $D=\bigcup T$.
Suppose that $T'=\{D_1',\ldots,D_s'\}$ consists of non-special
diagrams and $f:T\simeq T'$, where $f:D_j\mapsto D_j'$ for
$i=1,\ldots,s$. Our goal is to prove that $D_j=D_j'$ for any $j$.
From the elementary properties of innertial momentum we get:
\begin{equation}\label{eq04}
\sum_{j=1}^s\#D_j\Vert c(D_j)\Vert^2+i(D_j)
=\sum_{d\in D}\Vert d\Vert^2
=\sum_{j=1}^s\#D_j'\Vert c(D_j')\Vert^2+i(D_j')
\end{equation}
As $D_j$ are stable diagrams and $D_j'$ are non-special,
it follows that $i(D_j)\leq i(D_j')$.
Due to \eqref{eq04} we get $i(D_j)=i(D_j')$ for any $j=1,\ldots,s$.
From the stability of $D_j$ we get $\varphi_1(D_j)=\varphi_1(D_j')$
for any $j=1,\ldots,s$.

By assumption we can extend $\preceq$ to a partial ordering on $T$.
Without loss of generality we can assume that $D_1$ is minimal.
We define $m(\alpha):=\min\{\beta\in\NN:(\alpha,\beta)\in D_1\}$
for $\alpha\in\NN$ such that $\varphi_1(D_1)(\alpha)>0$,
and $m(\alpha):=0$ otherwise. Since the vertical sections of $D_1$ are
segments, $\varphi_1(D_1)(\alpha)=\varphi_1(D_1')(\alpha)$,
and $D_1$ is minimal we get:
\begin{multline}\label{eq05}
m(\alpha)+(m(\alpha)+1)+\ldots+(m(\alpha)+\varphi_1(D_1)(\alpha)-1)\\
=\frac{(\varphi_1(D_1)(\alpha)-1)\varphi_1(D_1)(\alpha)}{2}
+\varphi_1(D_1)(\alpha)\cdot m(\alpha)
\leq
\sum_{(\alpha,\beta)\in D_1'}\beta
\end{multline}
Summing up \eqref{eq05} for all possible $\alpha$, we get:
\begin{multline}\label{eq06}
\#D_1\cdot c_2(D_1)
=\sum_{\alpha\in\NN}\frac{(\varphi_1(D_1)(\alpha)-1)\varphi_1(D_1)(\alpha)}{2}
+\varphi_1(D_1)(\alpha)\cdot m(\alpha)\\
\leq\sum_{\alpha\in\NN}\sum_{(\alpha,\beta)\in D_1'}\beta
=\#D_1'\cdot c_2(D_1')
\end{multline}
Due to $\#D_1\cdot c_2(D_1)=\#D_1'\cdot c_2(D_1')$
we get equations in both \eqref{eq06} and \eqref{eq05}.
This implies $D_1=D_1'$.
According to the induction hypothesis we already know that
the tiling $T\setminus\{D_1\}$ is unique. Hence, from
$f:T\setminus\{D_1\}\simeq T'\setminus\{D_1\}$, we get
$D_2=D_2',\ldots,D_r=D_r'$.
\end{proof}

\section{The proof of Theorem \ref{main1}}

The proof shall be divided into several lemmas. The first
of these gives an explanation of why the linear system which
fullfils one of the conditions from Theorem \ref{main1}
is a special linear system.

\begin{uwaga}
One can consider using the Cremona transformation as a method
of verifing the specialty of these linear systems
(see for example \cite{Dum-Jar}). This is due to the fact
that every complete linear system of the form $\cL_{(d,e)}(m_1,\ldots,m_r)$
(over $\PP^1\times\PP^1$) is isomorphic to a linear system over $\PP^2$
(see \cite{CGG1} for more details).
\end{uwaga}

\begin{lemat}
The linear systems listed in the hypothesis of Theorem \ref{main1}
are special.
\begin{proof}
Throughout the proof we will refer to the following two
observations. We omit their proofs, as they are very simple.
\begin{obserwacja}\label{obs_aux1}
Consider a diagram $D$ and numbers $m_1,\ldots,m_r\geq 1$.
The following properties hold:
\begin{itemize}
 \item[$(i)$] if $\#D>\sum_i\binom{m_i+1}{2}$ then system
  $\cL_D(m_1,\ldots,m_r)$ is non-empty
  (i.e. it contains a non-zero polynomial)
 \item[$(ii)$] if we have equation in $(i)$ then the system
  $\cL_D(m_1,\ldots,m_r)$ is non-empty if and only if it is special
\end{itemize}
\end{obserwacja}

\begin{obserwacja}\label{obs_aux2}
Let $D$, $D'$ be diagrams, and $m_1,m_1',\ldots,m_r,m_r'\in\NN$
(some of them can be zero). Then, the following is true:
\begin{equation*}
\cL_D(m_1,\ldots,m_r)\cdot\cL_{D'}(m_1',\ldots,m_r')\s
\cL_{D+D'}(m_1+m_1',\ldots,m_r+m_r')
\end{equation*}
where $D+D'=\{d+d'\mid d\in D, d'\in D'\}$.
Furthermore, if the systems on the left are non-empty, then:
\begin{multline*}
\dim_L\cL_{D+D'}(m_1+m_1',\ldots,m_r+m_r')\\
\geq\max\{\dim_L\cL_D(m_1,\ldots,m_r),\dim_L\cL_D'(m_1',\ldots,m_r')\}
\end{multline*}
\end{obserwacja}

Let us begin with a system of form $(2)$.
W assume that $q+r=2k+1$ for some $k\geq 0$ and
claim that the following system is special:
$$
\cL_{(2,q+2r-1)}(2^{\times q},3^{\times r})
.
$$
Thanks to Observation \ref{obs_aux1}$(ii)$ it is sufficient
to state that the system is non-empty. From Observation \ref{obs_aux2}
it follows that:
$$
(\cL_{(1,k)}(1^{\times q},1^{\times r}))^2
\cdot \cL_{(0,r)}(0^{\times q},1^{\times r})\s
\cL_{(2,q+2r-1)}(2^{\times q},3^{\times r})
$$
The factors are non-empty due to Observation \ref{obs_aux1}$(i)$
as the coresponding diagrams have respectively $2\cdot(k+1)$ and $r+1$
elements. The thesis is now a consequence of Observation \ref{obs_aux2}.

The same type of reasoning can be applied for $(3.1)$, $(3.4)$ and $(4)$.
The following inclusions should be considered:
\begin{align*}
\cL_{(2,2n)}(0,2,2^{\times 2n})\cdot
\cL_{(1,n)}(0,1,1^{\times 2n})\;\s\; &
\cL_{(3,3n)}(0,3,3^{\times 2n})
\\
\cL_{(2,2n+2)}(2^{\times 2},2^{\times 2n+1})\cdot
\cL_{(1,n)}(0^{\times 2},1^{\times 2n+1})\;\s\; &
\cL_{(3,3n+2)}(2^{\times 2},3^{\times 2n+1})
\\
\cL_{(2,4)}(2^{\times 5})\cdot
\cL_{(2,1)}(1^{\times 5})\;\s\; &
\cL_{(4,5)}(3^{\times 5})
\end{align*}
The factors that contain base points of multiplicity $2$
are non-empty due to the case $(2)$.

We will need a more datailed estimation for $(3.2)$ and $(3.3)$.
Let us consider the following inclusions:
\begin{align*}
\cL_{(2,2n)}(0,2,2^{\times 2n})\cdot
\cL_{(1,n)}(0,0,1^{\times 2n})\:\s\; &
\cL_{(3,3n)}(0,2,3^{\times 2n})
\\
\cL_{(2,2n)}(0,2,2^{\times 2n})\cdot
\cL_{(1,n)}(1,0,1^{\times 2n})\;\s\; &
\cL_{(3,3n)}(1,2,3^{\times 2n})
\\
\cL_{(2,2n)}(0,0,2^{\times 2n+1})\cdot
\cL_{(1,n+1)}(0,0,1^{\times 2n+1})\;\s\; &
\cL_{(3,3n+1)}(0,0,3^{\times 2n+1})
\\
\cL_{(2,2n)}(0,1,2^{\times 2n+1})\cdot
\cL_{(1,n+1)}(0,1,1^{\times 2n+1})\;\s\; &
\cL_{(3,3n+1)}(0,1,3^{\times 2n+1})
\\
\cL_{(2,2n)}(1,1,2^{\times 2n+1})\cdot
\cL_{(1,n+1)}(1,1,1^{\times 2n+1})\;\s\; &
\cL_{(3,3n+1)}(1,1,3^{\times 2n+1})
\end{align*}
For each case we can easily compute the dimension
of the second factor, which equals respectively: $2,1,3,2,1$
(the systems containing points of multiplicity $1$ are alway non-special).
From Observation \ref{obs_aux2} we know that the dimension
of the system on the right side of the inclusion is at least: $2,1,3,2,1$,
which is more than the expected dimension. Therefore, all systems on the
right are special.

Finally, we move to the cases of $(0)$ and $(1)$.
Consider the system
$\cL_{(1,e)}(1^{\times p},2^{\times q},3^{\times r})$.
Suppose that $p+3q+5r\leq 2e+1$ and $r\geq 1$.
Let us denote
$M=M_{(1,e)}(1^{\times p},2^{\times q},3^{\times r})$
Our goal is to show that each minor of $M$ of the maximal rank is zero.

Evry $6$ points arbitrarily choosen from the diagram
$\{0,1\}\times\{0,\ldots,e\}$ are contained in two lines. Therefore,
it follows from Proposition \ref{prop02} that the rows of $M$
corresponding to a point of multiplicity $3$ (we assumed that there
is at least one such point) are linearly dependent.
Hence, if only $p+3q+6r\leq 2e+2$ (i.e. the number of rows does
not exceed the number of columns), the rank of $M$ can not be maximal.

Suppose that $p+3q+6r>2e+2$ (i.e. there are more rows than columns).
From $p+3q+5r\leq 2e+1$ it follows that among any $2e+2$ rows there
are at least $6$ rows coresponding to the same point of multiplicity $3$.
But, as before observed, these rows are linearly dependent.

Let us consider the system
$\cL_{(0,e)}(1^{\times p},2^{\times q},3^{\times r})$.
We assume that $p+2q+3r\leq e$ and $q\geq 1$ or $r\geq 1$.
Let $M=M_{(0,e)}(1^{\times p},2^{\times q},3^{\times r})$.
If the number of rows does not exceed the number of columns
we proceed as before. Otherwise,
from $p+2q+3r\leq e$ it follows that among any $e+1$ rows
there are at least $3$ rows corresponding to a point of multiplicity $2$
or at least $4$ rows corresponding to a point of multiplicity $3$.
We can apply the previous arguments to the former case.
To deal with the latter case, observe that given any $4$ rows
coresponding to a point of multiplicity $3$, at least one of
them is zero. Indeed, each monomial of the form $X^{\alpha}Y^{\beta}$,
where $\alpha\leq 1$, becomes zero after calculating the second
derivative with respect to $X$.
\end{proof}
\end{lemat}

\begin{uwaga}
In the following lemmas we prove non-specialty of a large
class of linear systems. What we actually show, is the following:

\textit{Given a system of the form
$\cL_{(d,e)}(1^{\times p},2^{\times q},3^{\times})$,
for which none of the conditions $(0)-(4)$ from Theorem \ref{main1}
hold, there exists a tiling that fulfills both the
hypothesis of Theorem \ref{prop08} and Theorem \ref{tw13}.}
\end{uwaga}

For $k,l\geq 1$ we shall use the notion
$k\times l:=\{0,\ldots,k-1\}\times\{0,\ldots,l-1\}$.
In several cases we will write
$\cL_{k\times l}(\ldots)$ instead of $\cL_{(k-1,l-1)}(\ldots)$.
It is obvious, that if a proper tiling is already constructed
for a system of the form $\cL_D(2^{\times q},3^{\times r})$
we can add some $1-$diagrams (see Example \ref{ex_sing})
so as to achive the
desired value of $p$. Therefore, for simplicity, we always
assume that $p=0$. The assumption from Theorem \ref{tw13}
concerning the possibility of extending relation $\preceq$
to a partial ordering, shall be alway fulfield.
This will follow immediately from Observation \ref{obs12}$(ii)$.
\begin{lemat}\label{lem_aux1}
For any $d,e\geq 5$ and $p,q,r\geq 0$, a system of the form
$\cL_{(d,e)}(1^{\times p},2^{\times q},3^{\times r})$
is non-special.
\begin{proof}
Let $k=d+1$, $l=e+1$. Without losing generality we can assume that
$kl-3<3q+6r<kl+6$ and $p=0$. We shall prove that it is sufficient to
consider $k,l<12$. Let us observe that any rectangular diagram
of height $6$ can be tiled with either $6-$diagrams or $3-$diagrams
by means of the schemes presented in Figure \ref{fig:6x}.
\begin{figure}
\epsfbox{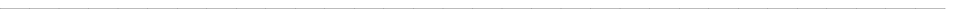}
\caption{
See Lemma \ref{lem_aux1}
\label{fig:6x}
}
\end{figure}
Now suppose, that $6\leq k\leq l$ and $l\geq 12$.
From $kl-3\leq 3p+6q$ it follows that $q\geq k$ or $p\geq 2k$.
If the former holds then we can tile a rectangle $k\times 6$
and reduce the problem to a smaller rectangle
$k\times (l-6)$, while still having $l-6\geq 6$ and $k(l-6)-3<3p+6(q-k)$.
We proceed in the same way in case of $p\geq 2k$.

Therefore, the problem of finding a proper tiling can
be reduced to a fininte number of cases, namely $k,l<12$
and $kl-3<3q+6r<kl+6$. Some possible schemes that covers
all neccessery constructions can be found in Figure \ref{fig:6_schemes}
at the end of this paper.
Because each $6-$diagram, apart from the ''triangular'' one, can
be divided into $3-$diagrams, even the use of only $6-$diagrams
satisfies all cases.

For $k=6$ and $k=9$
the ''triangular'' diagram has to be used (see Figure \ref{fig:6_schemes}).
To address this problem
one can find a way of covering a ''triangular'' diagram together with one
of its ''neighbours'' using four $3-$diagrams.
\end{proof}
\end{lemat}

\begin{lemat}\label{lem_aux2}
For any $e\geq 6$ and $p,q,r\geq 0$, a linear system
of the form
$\cL_{(4,e)}(1^{\times p},2^{\times q},3^{\times r})$
is non-special.
\begin{proof}
Let us denote $k=e+1$.
We proceed as in the proof of Lemma \ref{lem_aux1}.
We will reduce the problem
of finding appropriate tiling to a finite number of
cases, namely $k<19$.

We can assume that $p=0$ nad $5k-3<3q+6r$. Therefore, we get
$q\geq 10$ or $r\geq 10$. If $q\geq 10$ then the problem can be
reduced to the rectangle $5\times (k-12)$ as one can tile
the rectangle $5\times 12$ using the example presented
in Figure \ref{fig:5x}.
\begin{figure}
\epsfbox{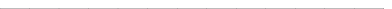}
\caption{
See Lemma \ref{lem_aux2}
\label{fig:5x}
}
\end{figure}
If $r<10$ and $q\geq 10$, we reduce to the rectangle $5\times (k-6)$.
The schemes that cover the case of $k<19$ can be found
in Figure \ref{fig:5_schemes}.
\end{proof}
\end{lemat}

\begin{lemat}\label{lem_aux4}
The system $\cL_{(4, 5)}(1^{\times p},2^{\times q},3^{\times r})$
is special if and only if $p=q=0$ and $r=5$.
\begin{proof}
The schemes presented in Figure \ref{fig:5x6} cover all possible cases.
\begin{figure}
\epsfbox{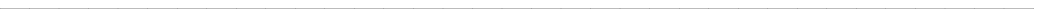}
\caption{
See Lemma \ref{lem_aux4}
\label{fig:5x6}
}
\end{figure}
\end{proof}
\end{lemat}

\begin{lemat}\label{lem_aux5}
Consider numbers $3\leq e$, $p,q,r\geq 0$
and let $n$ be such that $0\leq e-3n<3$.
Supposing that none of the following conditions hold
(see Theorem \ref{main1}):
\begin{itemize}
 \item[$(3.1)$] $e=3n$, $p=q=0$ and $r=2n+1$
 \item[$(3.2)$] $e=3n$, $p\leq 1$, $q=1$ and $r=2n$
 \item[$(3.3)$] $e=3n+1$, $p\leq 2$, $q=0$ and $r=2n+1$
 \item[$(3.4)$] $e=3n+2$, $p=0$, $q=2$ and $r=2n+1$
\end{itemize}
then the system
$\cL_{(3,e)}(1^{\times p},2^{\times q},3^{\times r})$
is non-special.
\begin{proof}
Observe that if $r\geq 2n+2$ then the rectangle $4\times (e+1)$
can be covered with $n+1$ rectangle $4\times 3$, and each such
rectangle can be tiled witch two $6-$diagrams.

When $r\leq 2n-1$ then one of the algorithms presented
in Figure \ref{fig:alg4x} can
be used to consturct a tiling that fulfills the hypothesis
of Theorem \ref{prop08}. The block on the right should be
used if the number of $6-$diagrams is even.
\begin{figure}
\epsfbox{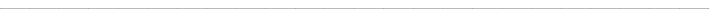}
\caption{
See Lemma \ref{lem_aux5}
\label{fig:alg4x}
}
\end{figure}
We still need to construct a tiling for $r=2n$ and $r=2n+1$.
Observe that using $n-1$ rectangles $4\times 3$, we reduce
the problem to $r=2$ or $r=3$ and $e+1\in\{4,5,6\}$.
Since we assumed that none of the conditions (3.1-4) hold,
one of the schemes in Figure \ref{fig:4x456} can be used to produce
the desired tiling.
\begin{figure}
\epsfbox{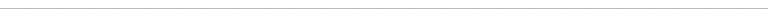}
\caption{
See Lemma \ref{lem_aux5}
\label{fig:4x456}
}
\end{figure}
\end{proof}
\end{lemat}

\begin{figure}
\epsfbox{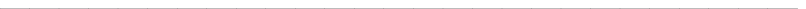}
\caption{
See Lemma \ref{lem_aux2}
\label{fig:5_schemes}
}
\end{figure}

\pagebreak

\begin{figure}
\epsfbox{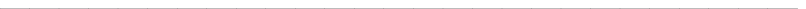}
\caption{
See Lemma \ref{lem_aux1}
\label{fig:6_schemes}
}
\end{figure}

\end{document}